\documentclass[a4paper,12pt]{article}

\usepackage{amsmath}
\usepackage{amssymb}
\usepackage{amsthm}

\theoremstyle{plain}
\newtheorem{theorem}{Theorem}[section]
\newtheorem{conjecture}[theorem]{Conjecture}
\newtheorem{corollary}[theorem]{Corollary}
\newtheorem{lemma}[theorem]{Lemma}
\newtheorem{proposition}[theorem]{Proposition}
\newtheorem*{claim*}{Claim}
\newtheorem*{problem*}{Problem}
\newtheorem*{conjecture*}{Conjecture}
\newtheorem*{theorem*}{Theorem}

\theoremstyle{definition}
\newtheorem{definition}[theorem]{Definition}
\newtheorem*{definition*}{Definition}

\newtheorem{example}[theorem]{Example}

\newcommand\RR{\mathbb{R}}

\newcommand\ad{\mathrm{ad}}
\newcommand\Ann{\mathrm{Ann}}

\newcommand\cA{\mathcal{A}}
\newcommand\cF{\mathcal{F}}
\newcommand\cH{\mathcal{H}}
\newcommand\cJ{\mathcal{J}}
\newcommand\cM{\mathcal{M}}

\newcommand\al{\alpha} 
\newcommand\bt{\beta} 
\newcommand\Dl{\Delta}
\newcommand\lm{\lambda} 
\newcommand\Lm{\Lambda}
\newcommand\Mu{\mathrm{M}} 

\newcommand\la{\langle}
\newcommand\ra{\rangle}

\newcommand{\A}{\mathrm{A}}
\newcommand{\B}{\mathrm{B}}
\newcommand{\C}{\mathrm{C}}

\newcommand\Aut{\mathrm{Aut}}

\renewcommand{\phi}{\varphi}

\title{On the structure of axial algebras} 
\author{S.M.S.~Khasraw\footnote{Department of Mathematics, College of Education, Salahaddin University-Erbil, Erbil, Kurdistan Region, Iraq, email: sanhan.khasraw@su.edu.krd}
\and
 J.~M\textsuperscript{c}Inroy\footnote{School of Mathematics, University of Bristol, Bristol, BS8 1TW, UK, and the Heilbronn Institute for Mathematical Research, Bristol, UK, email: justin.mcinroy@bristol.ac.uk}
 \and
 S.~Shpectorov\footnote{School of Mathematics, University of Birmingham, Edgbaston, Birmingham, B15 2TT, UK, email: S.Shpectorov@bham.ac.uk}}
\date{\today}

\begin{document}
\maketitle
\begin{abstract}
Axial algebras are a recently introduced class of non-associative algebra motivated by applications to groups and vertex-operator algebras. We develop the structure theory of axial algebras focussing on two major topics: (1) radical and simplicity; and (2) sum decompositions. 
\end{abstract}

\section{Introduction}

Axial algebras are a new class of non-associative algebra introduced by Hall, Rehren and Shpectorov \cite{hrs}. They axiomatise some key properties of vertex operator algebras (VOAs). VOAs were first introduced by physicists but particularly became of interest to mathematicians with Frenkel, Lepowsky and Meurman's \cite{FLM} construction of the moonshine VOA $V^\natural$ whose automorphism group is the Monster $M$, the largest sporadic finite simple group. 
The rigorous theory of VOAs was developed by Borcherds \cite{B86} and it was instrumental in his proof of the monstrous moonshine conjecture.

An \emph{axial algebra} is a commutative non-associative algebra $A$ generated by a set of axes $X$. These axes are idempotents whose adjoint action decomposes the algebra as a direct sum of eigenspaces and the multiplication of eigenvectors satisfies a certain fusion law. Jordan and Matsuo algebras are examples of axial algebras with one of the simplest (and strongest) fusion laws. A slight relaxation of this fusion law adds the Griess-Norton algebra for the Monster $M$ and other interesting examples.

Such axial algebras are of interest because the fusion law is $\mathbb{Z}_2$-graded. Hence, for an axis $a$, this induces a $\mathbb{Z}_2$-grading on the algebra and there is a natural involution $\tau_a$ associated to $a$. The group generated by the set of all such $\tau_a$, for $a \in X$, is called the \emph{Miyamoto group} and it is a subgroup of the automorphism group of $A$.

In this paper, we introduce an equivalence relation on sets of axes in an axial algebra. A set $X$ of axes is closed if it is closed under the action of the Miyamoto group $G$ defined by $X$; that is, $\bar X=X$, where $\bar{X} := X^G$. Two sets $X$ and $Y$ are equivalent if their closures $\bar{X}$ and $\bar{Y}$ are equal.  We say that a property of an axial algebra is \emph{stable} if it is invariant under equivalence of axes.  In this paper, we introduce several new properties of axial algebras and we show that they, and some existing well-known properties, are stable.  Firstly, we show that generation of axial algebras is stable.  That is, equivalent sets of axes generate the same algebra.  The Miyamoto group of an axial algebra is also stable.

We introduce the radical of an axial algebra $A$ with axes $X$ and show that it too is stable.

\begin{definition*}
The \emph{radical} $R(A, X)$ of $A$ with respect to the generating set of primitive axes $X$ is the unique largest ideal of $A$ containing no axes from $X$.
\end{definition*}

This gives us a way to split ideals of $A$, so that we may consider separately those which are contained in the radical and those which contain an axis.  We introduce the \emph{projection graph} on the set of axes $X$ and show how this determines which axes are contained in a proper ideal.

A \emph{Frobenius form} on an axial algebra is a non-zero (symmetric) bilinear form $(\cdot, \cdot)$ which associates with the algebra product. That is, $(a,bc) = (ab,c)$ for all $a,b,c \in A$. Currently, all known axial algebras admit such a form. Comparing our notion of the radical with that of the form, we have the following.

\begin{theorem*}
Let $A$ be a primitive axial algebra with a Frobenius form. Then the radical $A^\perp$ of the Frobenius form coincides with the radical $R(A, X)$ of $A$ if and only if $(a,a)\neq 0$ for all $a\in X$.
\end{theorem*}

In particular, when an axial algebra has a Frobenius form, we can use the above theorem as an easy way to find the radical $R(A, X)$.  We also give an application of the above theorem to show that all the Norton-Sakuma algebras apart from $2\B$ are simple.

In the second half of the paper, we discuss sum decompositions of axial algebras. It is clear from our definition of the radical that the annihilator $\Ann(A) \subseteq R(A, X)$. We show that if $A$ has a sum decomposition of (not necessarily axial) subalgebras $A = \sum_{i \in I} A_i$, where $A_i A_j = 0$ for $i \neq j$, and $\Ann(A) = 0$, then $A = \bigoplus_{i \in I} A_i$.  

\begin{theorem*}
Suppose that $A = \sum_{i \in I} A_i$ is a sum decomposition and let $X_i \subseteq X$ be the set of axes which are contained in $A_i$.  Then, $A = \sum_{i \in I} B_i$, where $B_i = \la \la X_i \ra\ra$ is an axial algebra.  Moreover, this decomposition into the sum of axial subalgebras is invariant under arbitrary changes of axes \textup{(}not just equivalence\textup{)}.
\end{theorem*}

This suggests the following definition.

\begin{definition*}
The \emph{non-annihilating graph} $\Delta(X)$ of an axial algebra $A$ with generating axes $X$ is the graph with vertex set $X$ and, for $a \neq b$, an edge $a \sim b$ if and only if $ab \neq 0$.
\end{definition*}

It is clear that if $A= \sum_{i \in I} A_i$ is a sum decomposition where the $A_i$ are axial algebras, then the corresponding $X_i$ are unions of connected components of $\Delta(X)$. It is natural to ask: is it not true that the finest sum decomposition of an axial algebra arises when each $X_i$ is a single connected component of $\Delta$? In particular, we make the following conjecture.

\begin{conjecture*}
The finest sum decomposition of an axial algebra $A$ of Monster type arises when each $X_i$ is just a single connected component of $\Dl$.
\end{conjecture*}

We show that the Miyamoto groups do indeed respect this decomposition of the axes.

\begin{theorem*}
Let $A$ be a $T$-graded axial algebra with $0 \in \mathcal{F}_{1_T}$ and the components of $\Delta(X)$ be $X_i$ for $i \in I$.  Then, $G(X)$ is a central product of the $G(X_i)$.
\end{theorem*}

However, for the algebra, the picture is more complicated and we give a partial result in this direction. In order to do so, we introduce a new concept.  A subspace $I \leq A$ is a \emph{quasi-ideal} if $I$ is invariant under multiplication with the axes $X$.  We show that a quasi-ideal is stable and $G(X)$-invariant.  The \emph{spine} of $A$ is defined as the quasi-ideal $Q(A, X)$ generated by the axes $X$. The algebra $A$ is called \emph{slender} if $A = Q(A, X)$.  We prove that the spine of an axial algebra is stable.  A fusion law $\cF$ is called \emph{Seress} if $0 \in \cF$ and $0 \star \lambda \subseteq \{\lambda \}$ for all $\lambda \in \cF$.  Our partial result is the following.

\begin{theorem*}
Let $A$ be an axial algebra with a Seress fusion law and $A_i = \la \la X_i \ra \ra$ be the axial subalgebra generated by the connected component $X_i$ of $\Dl$.  If all but possibly one $A_i$ are slender, then $A = \sum_{i \in I} A_i$.
\end{theorem*}

Recall that an axial algebra $A$ is \emph{$m$-closed} if $A$ is spanned by products in the axes of length at most $m$.  Note that the spine is spanned by products in the axes of the form $x_1(x_2( \dots (x_{k-1} x_k) \dots)$.  So, in particular, every $3$-closed algebra is slender.  Hence, we should expect the above result to apply to a large class of Seress axial algebras.

The paper is organized as follows. In Section \ref{sec:axialalgebras}, we recall the definition of axial algebras and review some basic properties. We discuss automorphisms and the Miyamoto group in Section \ref{sec:auto}.  Here, we also introduce equivalence of sets of axes, stability and show that generation of axial algebras and the Miyamoto group are stable. Section \ref{sec:radform} introduces the radical $R(A,X)$ and we show that it is stable. We also introduce the projection graph and use it to prove results about ideals.  The Frobenius form is introduced and we prove some important properties.  The main theorem in this section is that the radical of the form coincides with the radical of the algebra. In Section \ref{sec:decomp}, we discuss sum decompositions of axial algebras and show when they are direct.  Finally, in Section \ref{sec:NAgraph}, we introduce the non-annihilating graph $\Delta$ and our conjecture on connected components of $\Delta$.  We introduce quasi-ideals and the spine $Q(A, X)$ and show they are both stable.  Finally, we prove results about the decomposition with respect to $\Delta$.

\medskip

We would like to thank Jonathan I. Hall for useful comments.

\section{Axial algebras}\label{sec:axialalgebras}

\subsection{Fusion laws}\label{sec:fusionrules}
Throughout the paper $\mathbb{F}$ is an arbitrary field.

\begin{definition}
A \emph{fusion law} over $\mathbb{F}$ is a finite set $\cF$ of elements of $\mathbb{F}$ together with a symmetric 
map $\star \colon \cF\times\cF\rightarrow 2^\cF$. A single instance $\lambda \star \mu$ is called a \emph{fusion rule}.
\end{definition}

Since the values of $\star$ can be arranged in a symmetric square table, similar to 
a multiplication table, we sometimes call a fusion law a fusion table.  We will often abuse notation and just write $\cF$ for the fusion law $(\cF, \star)$.
\begin{figure}[ht]
\begin{center}
\begin{minipage}[t]{0.17\linewidth}
\renewcommand{\arraystretch}{1.5}
\begin{tabular}[t]{c||c|c}
 & 1 & 0 \\
\hline\hline
$1$ & $1$ & \\
\hline
$0$ & & $0$
\end{tabular}
\end{minipage}
\begin{minipage}[t]{0.25\linewidth}
\renewcommand{\arraystretch}{1.5}
\begin{tabular}[t]{c||c|c|c}
 & $1$ & $0$ & $\eta$ \\
\hline\hline
$1$ & $1$ & & $\eta$ \\
\hline
$0$ & & $0$ & $\eta$ \\
\hline 
$\eta$ & $\eta$ & $\eta$ & $1,0$
\end{tabular}
\end{minipage}
\begin{minipage}[t]{0.33\linewidth}
\renewcommand{\arraystretch}{1.5}
\begin{tabular}[t]{c||c|c|c|c}
 & $1$ & $0$ & $\alpha$ & $\beta$ \\
\hline\hline
$1$ & $1$ & & $\alpha$ & $\beta$ \\
\hline
$0$ & & $0$ & $\alpha$ & $\beta$ \\
\hline 
$\alpha$ & $\alpha$ & $\alpha$ & $1,0$ & $\beta$ \\
\hline 
$\beta$ & $\beta$ & $\beta$ & $\beta$ & $1,0,\alpha$
\end{tabular}
\end{minipage}
\end{center}
\caption{Fusion laws $\cA$, $\cJ(\eta)$, and $\cM(\al,\bt)$}
\label{fusion laws}
\end{figure}
In Figure \ref{fusion laws}, we see three examples of fusion laws that have 
appeared in the literature.  In the tables, we abuse notation by neglecting 
to write the set symbols.  We also leave the entry blank to mean the empty set.

In the first example, the set $\cF=\cA$ consists of 
just the elements $1$ and $0$ of $\mathbb{F}$. Hence this is defined over 
every field $\mathbb{F}$.  In the second example $\cF=\cJ(\eta)=\{1,0,\eta\}$, 
where $\eta\in \mathbb{F}$ and $1\neq\eta\neq 0$. 
So this can be defined for any field $\mathbb{F}$ except $\mathbb{F}_2$.
Similarly, in the third example, $\cF=\cM(\al,\bt)=\{1,0,\al,\bt\}$, where 
$\al,\bt\in \mathbb{F}$, $\al,\bt\not\in\{1,0\}$, and $\al\neq\bt$. Hence, for this to 
make sense, the field $\mathbb{F}$ must have at least four elements.

Given a fusion law $\cF$ and a subset $\cH\subseteq\cF$, $\cF$ induces fusion rules on $\cH$ by defining
\[
\lm\circ\mu := (\lm\star\mu)\cap\cH \qquad \mbox{for } \lm,\mu\in\cH
.\]
We call such a fusion law on $\cH$ a sublaw or \emph{minor} of $\cF$. We say that 
$\cH$ is \emph{exact} if $\lm\circ\mu=\lm\star\mu$ for all $\lm,\mu\in\cH$, that 
is, if $\cH$ is \emph{closed} for $\star$. For example, $\cA$ is an exact minor 
of both $\cJ(\eta)$ and $\cM(\al,\bt)$. We see that $\cJ(\eta)$ is a minor of 
$\cM(\al,\bt)$ in two ways: when $\eta=\al$ and when $\eta=\bt$. However, it is 
exact only when $\eta=\al$.

\subsection{Axes and axial algebras}

Let $A$ be a commutative non-associative (that is, not necessarily associative) 
algebra over $\mathbb{F}$. The adjoint of $a \in A$, denoted by $\ad_a$, is the linear 
endomorphism of $A$ defined by $b\mapsto ab$ for $b\in A$. For $\lm\in \mathbb{F}$, let 
$A_\lm(a)$ denote the $\lm$-eigenspace of $\ad_a$. That is, 
$A_\lm(a)=\{b\in A : ab=\lm b\}$. Clearly, $A_\lm(a)\neq 0$ if and only 
if $\lm$ is an eigenvalue of $\ad_a$. For $\Lm\subseteq \mathbb{F}$, we write 
$A_\Lm(a)=\bigoplus_{\lm\in\Lm}A_\lm(a)$.

\begin{definition}
For a fusion law $\cF$, an element $a\in A$ is an \emph{$\cF$-axis} if the 
following hold:
\begin{enumerate}
\item [{\rm (A1)}] $a$ is an idempotent; that is, $a^2=a$;
\item [{\rm (A2)}] $\ad_a$ is semisimple and all eigenvalues of $\ad_a$ are in 
$\cF$; that is, $A=A_\cF(a)$;
\item [{\rm (A3)}] the fusion law $\cF$ controls products of eigenvectors: namely,
\[
A_\lm(a)A_\mu(a)\subseteq A_{\lm\star\mu}(a) \qquad \mbox{for }\lm,\mu\in\cF.
\]
\end{enumerate}   
\end{definition}

Note that, $a$ being an idempotent, $1$ is an eigenvalue of $\ad_a$. For this 
reason, we will always assume that $1\in\cF$.  We also allow for the possibility 
that $A_\lambda(a)$ is $0$ for some $\lambda \in \mathcal{F}$.

\begin{definition}
An $\cF$-axis $a$ is \emph{primitive} if $A_1(a)=\la a\ra$.  
\end{definition}

If $a$ is a primitive axis then $A_1(a)A_\lm(a)=A_\lm(a)$, for all $\lm\neq 0$, and 
$A_1(a)A_0(a)=0$. Therefore, for primitive axes, we only need to consider fusion 
rules $\cF$ satisfying $1\star\lm=\{\lm\}$ for $\lm\neq 0$ and $1\star 0=\emptyset$, 
provided that $0\in\cF$. All three fusion laws in Figure \ref{fusion laws} 
possess this property.

\begin{definition}
An \emph{$\cF$-axial algebra} is a pair $A = (A, X)$, where $A$ is a commutative non-associative algebra generated by the set $X$ of $\cF$-axes.  An axial algebra $(X, A)$ is \emph{primitive} if each axis in $X$ is primitive.
\end{definition}

We will usually abuse notation and just refer to $A$ as being an axial algebra without making reference to $\mathcal{F}$ and $X$ where they are clear.  We will also often consider just primitive axial algebras and so we will often skip this adjective.

It is easy to show that associative axial algebras are the same as $\cA$-axial 
algebras. Furthermore, these are exactly the direct sum algebras 
$\mathbb{F}\oplus \ldots\oplus \mathbb{F}$.

Given any $3$-transposition group $(G, D)$, one can define a Matsuo algebra which has basis given by the elements of $D$ and multiplication depending on the order of the product of the involutions.  These are examples of $\cJ(\eta)$-axial algebras.  For more details see the text before Example \ref{ex:Matsuo}.

Every idempotent in a Jordan algebra satisfies the fusion law $\cJ(\frac{1}{2})$. 
This is known as the \emph{Peirce decomposition}. Hence Jordan algebras generated 
by primitive idempotents are examples of $\cJ(\frac{1}{2})$-axial algebras. 

Finally, the $196,884$-dimensional real Griess-Norton algebra, whose automorphism 
group is the sporadic simple Monster group $M$, is an example of an axial algebra with fusion law $\cM(\frac{1}{4},\frac{1}{32})$.  We call an axial algebra with this fusion law an axial algebra of \emph{Monster type}.

Moreover, the Griess-Norton algebra is an example of a Majorana algebra. \emph{Majorana algebras}, as introduced by Ivanov  \cite[Section 8.6]{ivanov}, are axial algebras of Monster type over $\mathbb{F}=\RR$, satisfying certain additional properties. Historically, Majorana algebras were precursors of axial algebras.


\section{Automorphisms}\label{sec:auto}

\subsection{Axis subgroup}

The fusion laws $\cF$ which particularly interest us are those where the axes lead to automorphisms of the algebra. Let us extend the 
operation $\star$ to arbitrary subsets $\Lm$ and $\Mu$ of the fusion table $\cF$ via 
$\Lm\star\Mu:=\cup_{\lm\in\Lm,\mu\in\Mu}\lm\star\mu$.  

\begin{definition}
Suppose $T$ is an abelian group. A \emph{$T$-grading} of a fusion table $\cF$ is a partition 
$\cF=\cup_{t\in T}\cF_t$ of $\cF$ satisfying $\cF_s\star\cF_t\subseteq\cF_{st}$ 
for all $s,t\in T$.
\end{definition}

Note that we allow the possibility that some part $\cF_t$ is the empty set.
Suppose that $\cF$ is $T$-graded and let $A$ be an $\cF$-axial algebra. For $t\in T$, 
we set $A_t(a)=A_{\cF_t}(a)=\bigoplus_{\lm\in\cF_t}A_\lm(a)$. Clearly, we have 
$A=\bigoplus_{t\in T}A_t(a)$. Note that it follows from the above definition that 
$A_t(a)A_s(a)\subseteq A_{ts}(a)$, that is, we have a $T$-grading of the algebra $A$ 
for each axis $a$.  Note that, since $\cF_t$ may be empty, $A_t$ may be $0$ for some $t \in T$.

Let $T^\ast$ be the group of linear characters of $T$ over $\mathbb{F}$, that is, the set of all 
homomorphisms from $T$ to the multiplicative group of $\mathbb{F}$. For an axis $a$ and 
$\chi\in T^\ast$, consider the linear map $\tau_a(\chi) \colon A\to A$ defined by 
\[
u \mapsto \chi(t) u \qquad \mbox{for } u \in A_t(a)
\]
and extended linearly to $A$. Since $A$ is $T$-graded, this map $\tau_a(\chi)$ is an automorphism 
of $A$. Furthermore, the map sending $\chi$ to $\tau_a(\chi)$ is a homomorphism from 
$T^\ast$ to $\Aut(A)$.

\begin{definition}
We call the image $T_a$ of the map $\chi \mapsto \tau_a(\chi)$, the \emph{axis subgroup} of $\Aut(A)$ corresponding to $a$.
\end{definition}

Usually, $T_a$ is a copy of $T^\ast$, but occasionally, when some subspaces $A_t(a)$ are 
trivial, $T_a$ can be isomorphic to a factor group of $T^\ast$ over a non-trivial subgroup.

We will often consider fusion laws where $T = C_2$.  If $\mathrm{char}(\mathbb{F}) = 2$, then $T^* = 1$ and we get no automorphisms.  So, we will normally assume that $\mathrm{char}(\mathbb{F}) \neq 2$ when $T = C_2$.  In this case, $T^* = \{ \chi_1, \chi_{-1} \}$ where $\chi_1$ is the trivial character.  The automorphism $\tau_a(\chi_{-1})$ is (usually) non-trivial and we will denote it by $\tau_a$.  Then $T_a = \langle \tau_a \rangle \cong C_2$.  We will also write $C_2 = \{+, -\}$.

Indeed, among our examples of fusion laws in Figure \ref{fusion laws}, the fusion tables 
$\cJ(\eta)$ and $\cM(\al,\bt)$ are $C_2$-graded. The 
grading for $\cJ(\eta)$ is given by $\cJ(\eta)_+=\{1,0\}$ and $\cJ(\eta)_-=\{\eta\}$. Whereas for 
$\cM(\al,\bt)$, the grading is given by $\cM(\al,\bt)_+=\{1,0,\al\}$ and 
$\cM(\al,\bt)_- = \{\bt\}$. Hence in these cases the axis subgroups are of order $2$ 
(or $1$ if $A_-(a)=0$).

Recall that the Griess-Norton algebra $A$ is an axial algebra with fusion law $\cM(\frac{1}{4}, \frac{1}{32})$.  For an axis $a\in A$, the subgroup $T_a=\la\tau_a\ra$ has order two.  Here the involutions $\tau_a$ belong 
to the conjugacy class $2A$ in the Monster $M$. Furthermore, the mapping $a\mapsto\tau_a$ 
is a bijection between the set of all axes of $A$ and the class $2A$.

Recall now that every axial algebra $A$ comes with a set of generating axes $X$. 
In the following definition we slightly relax conditions on $X$ by allowing it to be 
an arbitrary set of axes from $A$.

\begin{definition}
The \emph{Miyamoto group} $G(X)$ of $A$ with respect to the set of 
axes $X$ is the subgroup of $\Aut(A)$ generated by the axis subgroups $T_a$, $a\in X$.
\end{definition}

Since the $2A$ involutions generate the Monster, for the Griess-Norton algebra, we have $G(X) = M$ where $X$ is the set of $2A$-axes.

\subsection{Closed sets of axes}

If $a$ is an axis and $g\in\Aut(A)$, then $a^g$ is again an axis. Indeed, 
it is easy to check that $A_\lm(a^g)=A_\lm(a)^g$ and, hence, for $\lm,\mu\in\cF$, 
we have $A_\lm(a^g)A_\mu(a^g)=A_\lm(a)^g A_\mu(a)^g=
(A_\lm(a)A_\mu(a))^g\subseteq A_{\lm\star\mu}(a)^g=A_{\lm\star\mu}(a^g)$. 

\begin{definition}
A set of axes $X$ is \emph{closed} if $X^\tau=X$ for all $\tau\in T_a$ with $a\in X$. 
Equivalently, $X^{G(X)}=X$.
\end{definition}

It is easy to see that the intersection of closed sets is again closed and so every $X$ 
is contained in the unique smallest closed set $\bar X$ of axes. We call $\bar X$ the 
\emph{closure} of $X$.

\begin{lemma}\label{barX}
For a set of axes $X$, we have that $\bar X=X^{G(X)}$ and, furthermore, $G(\bar X)=G(X)$.
\end{lemma}

\begin{proof}
Since $X\subseteq\bar X$, we have that $G(X)\leq G(\bar X)$. Hence $X^{G(X)}\subseteq 
\bar X^{G(\bar X)}=\bar X$. To show the reverse inclusion, it suffices to prove that 
$X^{G(X)}$ is closed.

Suppose that $b\in X^{G(X)}$. Then $b=a^g$ for some $a\in X$ 
and $g\in G(X)$. Note that $\tau_{a^g}(\chi)=\tau_a(\chi)^g$ and so 
$T_b = T_{a^g}=T_a^g$. Since $T_a\leq G(X)$ and $g\in G(X)$, we have that $T_b=
T_a^g\leq G(X)^g=G(X)$. Hence, $G(X^{G(X)})=G(X)$. Clearly, $X^{G(X)}$ 
is invariant under $G(X)=G(X^{G(X)})$. This means that $X^{G(X)}$ is closed, proving 
that $\bar X=X^{G(X)}$ and also $G(\bar X)=G(X^{G(X)})=G(X)$.
\end{proof}

Turning again to the example of the Griess-Norton algebra, it is well-known that the Monster
$M$ can be generated by three $2A$ involutions, say, $\tau_a$, $\tau_b$, and $\tau_c$, 
for axes $a,b,c\in A$. Setting $X=\{a,b,c\}$, we have that $G(X)=\la T_a,T_b,T_c\ra=
\la\tau_a,\tau_b,\tau_c\ra=M$. Hence $\bar X=X^{G(X)}=X^M$ is the set of all axes of $A$, 
since $\{\tau_a,\tau_b,\tau_c\}^M$ is clearly all of $2A$. (We use the fact that the map 
sending an axis to the corresponding $2A$ involution is bijective.) So here $\bar X$ (of size approximately $9.7 \times 10^{19}$) is huge 
compared to the tiny $X$.

\begin{definition}
We say that sets $X$ and $Y$ of axes are \emph{equivalent} \textup{(}denoted $X\sim Y$\textup{)} if 
$\bar X=\bar Y$.
\end{definition}

Clearly, this is indeed an equivalence relation on sets of axes.

\begin{definition}
A property of an axial algebra is called \emph{stable} if it is invariant under equivalence of axes.
\end{definition}

In this paper, we will show that several properties of axial algebras are stable.  Lemma \ref{barX} gives us the first of these.

\begin{corollary}
The Miyamoto group of an axial algebra is stable.
\end{corollary}

Since $\bar X=X^{G(X)}$ and, similarly, $\bar Y=Y^{G(Y)}$, we have the following.

\begin{lemma}
Sets $X$ and $Y$ of axes are equivalent if and only if the following two conditions hold:
\begin{enumerate}
\item[$1.$] $G := G(X) = G(Y)$
\item[$2.$] Every $x\in X$ is $G$-conjugate 
to some $y\in Y$ and, vice versa, every $y\in Y$ is $G$-conjugate to some $x\in X$. \qed
\end{enumerate}
\end{lemma}

\subsection{Invariance}

Let $a \in X$ be an axis and $W$ be a subspace of $A$ invariant under the action of $\ad_a$. 
Since $\ad_a$ is semisimple on $A$, it is also semisimple on $W$, and so 
$W=\bigoplus_{\lm\in\cF}W_\lm(a)$, where $W_\lm(a)=W\cap A_\lm(a)=\{w\in W : aw=\lm w\}$. 

Let us note the following important property of axis subgroups $T_a$.

\begin{lemma} \label{invariant}
For an axis $a$, if a subspace $W\subseteq A$ is invariant under $\ad_a$ then $W$ is 
invariant under every $\tau_a(\chi)$, $\chi\in T^\ast$. \textup{(}That is, $W$ is invariant under 
the whole $T_a$.\textup{)}
\end{lemma}

\begin{proof}
We have already observed that if $W$ is invariant under $\ad_a$ then 
$W=\bigoplus_{\lm\in\cF}W_\lm(a)$. Recall that $W_\lm(a)$ is a subspace of $A_\lm(a)$. Since 
$\tau=\tau_a(\chi)$ acts on $A_\lm(a)$ as a scalar transformation, it leaves invariant every 
subspace of $A_\lm(a)$. In particular, $W_\lm(a)^\tau=W_\lm(a)$ for every $\lm$, and so 
$W^\tau=W$.
\end{proof}

For example, ideals of $A$ are invariant under $\ad_a$ for all axes $a$.  Hence we have the following:

\begin{corollary}\label{idealGinvariant}
Every ideal $I$ of $A$ is $G(X)$-invariant for any set of axes $X$ in $A$.
\end{corollary}

Let us now prove the following important property. We denote by $\la\la X\ra\ra$ the 
subalgebra of $A$ generated by the set of axes $X$.

\begin{theorem} \label{same subalgebra}
Suppose that $X\sim Y$. Then $\la\la X\ra\ra=\la\la Y\ra\ra$. In particular, if $X$ 
generates $A$ then so does $Y$.  Hence, generation of axial algebras is stable.
\end{theorem}

\begin{proof}
Let $B=\la\la X\ra\ra$ and $C=\la\la Y\ra\ra$. Note that $B$ is invariant under $\ad_a$ 
for every $a\in X$. Hence $B$ is $G(X)$-invariant. Clearly, this means that 
$\bar X=X^{G(X)}\subseteq B$. Therefore, $Y\subseteq\bar Y=\bar X\subseteq B$, proving 
that $C\subseteq B$. Symmetrically, also $B\subseteq C$, and so $B=C$.
\end{proof}

We note that the converse does not hold.  That is, there exist sets of axes $X$ and $Y$ 
which are inequivalent, but which both generate the same axial algebra $A$.  For example, 
there is an axial algebra of dimension $9$ which is generated by a closed set of $6$ axes 
(and has shape $3\textrm{C}2\textrm{A}$ and Miyamoto group $S_4$) 
\cite[Table 40]{algorithm}.  However, it is also generated (in fact, spanned by) a 
closed set of $9$ axes.  Since both sets are closed but of different sizes, they are 
clearly inequivalent.

If $A$ is the Griess-Norton algebra and $a$, $b$, and $c$ are axes such that 
$M=\la\tau_a,\tau_b,\tau_c\ra$. Setting $B=\la\la a,b,c\ra\ra$, we see that $B$ is 
invariant under $M$. Since we have a bijection between axes and the involutions from 
$2A$, all axes are conjugate under $M$. This shows that $B$ contains all axes from $A$, 
that is, $B=A$, since $A$ is generated by axes. We have shown that 
$A=\la\la a,b,c\ra\ra$, which means that, despite its large dimension, $A$ can be 
generated by just three axes.


\section{Ideals, the radical and the Frobenius form}\label{sec:radform}

Throughout this section, suppose $A$ is an axial algebra with fusion law $\cF$ over a field $\mathbb{F}$ and let $X$ be the set of primitive axes which generate $A$.

\subsection{The radical}

Recall that if $W$ is a subspace invariant under the action of $\ad_a$ for an axis $a$ then $W=\bigoplus_{\lm\in\cF}W_\lm(a)$, where $W_\lm(a)=W\cap A_\lm(a)$.

\begin{lemma}\label{axisinvariantsubspace}
Let $a \in X$ be a primitive axis and $W$ be a subspace of $A$ invariant under the action of $\ad_a$.  Then, $a \in W$ if and only if $W_1(a)=W\cap A_1(a)$ is not $0$.
\end{lemma}
\begin{proof}
Since $a$ is primitive, $A_1(a)=\la a\ra$ is $1$-dimensional. In particular, we have the dichotomy: 
either $a\in W$ and $W_1(a) = A_1(a) = \langle a \rangle$, or $W_1(a)=0$, and so $W\subseteq A_{\cF\setminus\{1\}}(a)$.
\end{proof}

In particular, the above lemma holds for ideals.  We begin by considering those ideals which do not contain any axes.

\begin{definition}
The \emph{radical} $R(A,X)$ of $A$ with respect to the generating set of primitive axes $X$ is 
the unique largest ideal of $A$ containing no axes from $X$.
\end{definition}

Abusing notation, we will drop either $A$ or $X$ where it is clear from context.  By Lemma \ref{axisinvariantsubspace}, an ideal, which clearly is invariant under the action of $\ad_a$ for all $a\in X$, 
contains no axes from $X$ if and only if it is contained in 
$\cap_{a\in X}A_{\cF\setminus\{1\}}(a)$. Clearly, the sum of all such ideals is again an ideal 
not containing any axes from $X$, so it is in the radical.  Hence there is indeed a unique largest ideal with the above property.

The radical $R(A,X)$ of an axial algebra $A$ is defined with respect to a given generating 
set of axes $X$. What if we take a different generating set?

\begin{theorem}
If $X \sim Y$ are two equivalent sets of axes, then $R(A, X) = R(A, Y)$.  That is, the radical of an axial algebra is stable.
\end{theorem}
\begin{proof}
It suffices to show that $R(A, X)=R(A, \bar X)$. 
Clearly, the ideal $R(\bar X)$ does not contain any axis from $X$, and so 
$R(\bar X)\subseteq R(X)$. Conversely, by Corollary \ref{idealGinvariant}, every ideal of $A$ is 
invariant under $G(X)$. Since $R(X)$ contains no axis from $X$, it follows that 
$R(X)$ contains no axis from $X^{G(X)}=\bar X$. So $R(X)\subseteq R(\bar X)$ and therefore $R(\bar X)=R(X)$. 

If $Y$ is equivalent to $X$, by Theorem \ref{same subalgebra}, $Y$ also generates $A$ and so $R(Y)$ is defined. Furthermore, 
$R(Y)=R(\bar Y)=R(\bar X)=R(X)$.
\end{proof}

This shows that our notion of the radical behaves well under the natural changes of generating sets of axes.


\subsection{Frobenius form}
Sometimes an $\cF$-axial algebra also admits a bilinear form which behaves well with respect to the multiplication in the algebra.

\begin{definition}
A \emph{Frobenius} form on an $\cF$-axial algebra $A$ is a \textup{(}non-zero\textup{)} bilinear form $(\cdot , \cdot) \colon A \times A \to \mathbb{F}$ which associates with the algebra product.  That is,
\[
(a,bc)=(ab,c) \qquad \mbox{for all } a,b,c\in A.
\]
\end{definition}

Note that we do not place any restriction on the value of $(a,a)$ for axes $a \in X$.  This differs from definitions given in previous papers.  However, several key facts still hold.  A Frobenius form is still necessarily symmetric \cite[Proposition 3.5]{hrs}. We also have the following important property:

\begin{lemma} \label{orthogonal}
For an axis $a$, the direct sum decomposition $A=\bigoplus_{\lm\in\cF}A_\lm(a)$ is orthogonal 
with respect to every Frobenius form $(\cdot , \cdot)$ on $A$.
\end{lemma}

\begin{proof}
Suppose $u\in A_\lm(a)$ and $v\in A_\mu(a)$ for $\lm\neq\mu$. Then $\lm(u,v)=(\lm u,v)=
(au,v)=(ua,v)=(u,av)=(u,\mu v)=\mu(u,v)$. Since $\lm\neq\mu$, we conclude that $(u,v)=0$.
\end{proof}

Let $a$ be a primitive axis.  Then we may decompose $u \in A$ with respect to $a$ as 
$u = \bigoplus_{\lambda \in \cF} u_\lambda$, where $u_\lambda \in A_\lambda(a)$.  We call 
$u_\lambda$ the \emph{projection} of $u$ onto $A_\lambda(a)$.  Focusing on the projection 
$u_1$, as $a$ is primitive, $u_1 = \phi_a(u)a$ for some $\phi_a(u)$ in $\mathbb{F}$.  It 
is easy to see that $\phi_a$ is linear in $u$.

\begin{lemma}\label{frobprim}
Let $(\cdot, \cdot)$ be a Frobenius form on a primitive axial algebra $A$.  Then,
$(a,u) = \phi_a(u)(a,a)$ for any axis $a \in X$ and $u \in A$.
\end{lemma}

\begin{proof}
We decompose $u = \bigoplus_{\lambda \in \cF} u_\lambda$ with respect to $a$, where 
$u_\lambda \in A_\lambda(a)$.  Now, by Lemma \ref{orthogonal}, 
$(a,u) = (a,  \bigoplus_{\lambda \in \cF} u_\lambda) = (a, u_1) = \phi_a(u)(a,a)$.
\end{proof}

Let us now explore the connection of the Frobenius form to the radical of $A$.  We write 
$A^\perp$ for the radical of the Frobenius form; that is,
\[
A^\perp=\{u\in A : (u,v)=0\mbox{ for all }v\in A\}
\]

\begin{lemma} \label{Frobenius radical}
The radical $A^\perp$ is an ideal of $A$. 
\end{lemma}

\begin{proof}
If $u\in A^\perp$ and $v,w\in A$, then $(uv,w)=(u,vw)=0$ and so $uv\in A^\perp$. Since 
$( \cdot, \cdot)$ is also bilinear, $A^\perp$ is an ideal.  
\end{proof}

It follows from Lemma \ref{frobprim} that a primitive axis $a$ is contained 
in $A^\perp$ if and only if $(a,a)=0$. Therefore, $A^\perp$ contains no axes from the 
generating set $X$ if and only if $(a,a)\neq 0$ for all $a\in X$.  The following is a 
generalisation of Proposition 2.7 in \cite{jordan}.

\begin{theorem} \label{radical of form}
Let $A = (A, X)$ be a primitive axial algebra with a Frobenius form. Then, the radical 
$A^\perp$ of the Frobenius form coincides with the radical $R(A, X)$ of $A$ if and only 
if $(a,a)\neq 0$ for all $a\in X$.
\end{theorem}

\begin{proof}
Let $R=R(A, X)$. If $A^\perp=R$ then $A^\perp$ contains no axes from $X$, and so, by 
Lemma \ref{frobprim}, we have that $(a,a)\neq 0$ for all $a\in X$. 

Conversely, suppose that $(a,a)\neq 0$ for all $a\in X$. Then $A^\perp$ contains no 
axes from $X$. Hence $A^\perp\subseteq R$. It remains to show that $R\subseteq A^\perp$, 
that is, that $R$ is orthogonal to the entire $A$. Since $X$ generates $A$, the algebra 
$A$ is linearly spanned by all (non-associative) products $w$ of the axes from $X$. Hence 
we just need to show that $R$ is orthogonal to each product $w$. We prove this property 
by induction on the length of the product $w$.

If the length of $w$ is one then $w=a$ is an axis from $X$. Since $a\not\in R$, we have 
that $R\subseteq A_{\cF\setminus\{1\}}(a)$, which by Lemma \ref{orthogonal} means that 
$R$ is orthogonal to $w$, as claimed. Now suppose that the length of $w$ is at least two. 
Then $w=w_1w_2$ for products $w_1$ and $w_2$ of shorter length. By the inductive 
assumption, we know that $R$ is orthogonal to both $w_1$ and $w_2$. Therefore, 
$(w,R)=(w_1w_2,R)=(w_1,w_2R)=0$, as $R$ is an ideal. So, $R\subseteq A^\perp$ and hence
$R=A^\perp$.
\end{proof}  

It is often additionally required that the Frobenius form satisfy $(a,a)=1$ for each axis 
$a$. In view of Lemma \ref{frobprim}, we call the Frobenius form satisfying $(a,a)=1$ for 
all generating axes $a$ the \emph{projection form}.  We will see later that the 
projection form, when it exists, is unique.

The existence of a projection form is included in the axioms of Majorana algebras by 
Ivanov. He further requires the projection form to be 
positive-definite. (Recall that Majorana algebras are defined over $\mathbb{F}=\RR$.) In 
particular, we have the following.

\begin{corollary}\label{Majalg}
Every Majorana algebra has trivial radical; that is, every non-zero ideal contains 
one of the generating primitive axes.
\end{corollary}

\begin{proof}
Indeed, since the Frobenius form is positive definite, we have that $(u,u)>0$ for every 
$u\neq 0$. In particular, this is true for axes, and so, by Theorem \ref{radical of form}, 
the radical of the algebra is the same as the radical of the Frobenius form, which is 
zero.
\end{proof}

For $\cJ(\eta)$-axial algebras, which are called \emph{axial algebras of Jordan type 
$\eta$}, we do not need to assume the existence of a projection form.  Every axial algebra 
of Jordan type automatically admits a projection form \cite{primitivejordan}. Hence, we 
can state the following.

\begin{corollary}
The radical of every algebra of Jordan type coincides with the radical of its projection 
form.
\end{corollary}

We wish to give an example, but first we must define the class of Matsuo algebras.  For 
any group of $3$-transpositions $(G, D)$, we define the \emph{Matsuo algebra} $A$ with 
respect to $(G, D)$ which has basis $D$ and multiplication given by
\[
ab = \begin{cases} a & \mbox{if } a = b \\
 0 & \mbox{if } o(ab)=2 \\
 \frac{\eta}{2}(a + b - c) & \mbox{if } o(ab) = 3 \mbox{, where } c = a^b = b^a
\end{cases}
\]
(Clearly here the field should not be of characteristic two.)  By 
\cite[Theorem 1.5]{jordan}, all Matsuo algebras with $\eta \neq 0,1$ are examples of axial 
algebras of Jordan type $\eta$. It can be seen that the projection Frobenius form for $A$ 
is given by
\[
(a,b) = \begin{cases} 1 & \mbox{if } b = a\\
0 & \mbox{if } o(ab)=2 \\
 \frac{\eta}{2} & \mbox{if } o(ab) = 3
\end{cases}
\]
Using the basis $D$, the form has Gram matrix
\[
F = I + \tfrac{\eta}{2} M
\]
where $I$ is the identity matrix and $M$ is the adjacency matrix of the non-commuting 
graph on $D$.  The form has a radical precisely when $F$ is not of full rank.  From the 
above equation for $F$, we see this occurs if and only if $\eta = -\frac{2}{\lambda}$ 
for a non-zero eigenvalue $\lambda$ of $M$. Furthermore, the radical of the form 
coincides with the $\lambda$-eigenspace of $F$. For example, the valency $\kappa$ of 
the non-commuting graph is an eigenvalue of $M$ and the corresponding eigenspace is 
$1$-dimensional spanned by the all-one vector. Hence, when $\eta=-\frac{2}{\kappa}$, 
the projection form has $1$-dimensional radical.

In general, since $M$ has finitely many eigenvalues, there are only finitely many values 
of $\eta$, for which the Frobenius form on the Matsuo algebra has a non-zero 
radical.

Here is a complete example.

\begin{example}\label{ex:Matsuo}
The group $G = S_5$ with the conjugacy class $D$ of transpositions is a $3$-transposition 
group and so leads to a Matsuo algebra $A$. By calculation (for example, see 
\cite{Fischer}), $M$ has eigenvalues $6, 1, -2$.  The value $\lambda=-2$ leads to 
$\eta = 1$, so this may be discarded.  When $\lambda = \kappa = 6$, $\eta=-\frac{1}{3}$ 
and the radical is spanned by the element $r:= \sum_{a \in D} a$.  When $\lambda = 1$, 
$\eta = -2$ and we have a $4$-dimensional radical spanned by elements of the form
\[
(i,j) + (i,k) + (i,l) - (m, j) - (m, k) - (m, l)
\]
where $\{i,j,k,l,m\} = \{1,2,3,4,5\}$.
\end{example}


\subsection{Ideals and the projection graph}

Having considered ideals which do not contain any axes, we now turn our attention to 
ideals $I$ that do contain an axis $a$.  What other axes does $I$ contain?  Suppose $b$ 
is another axis of $A$ which is primitive and let $a$ have decomposition 
$a = \bigoplus_{\lambda \in \mathcal{F}} a_\lambda$, where $a_\lambda \in A_\lambda(b)$.  
Since $I$ is invariant under $\ad_b$, we have $I=\oplus_{\lm\in\cF}I_\lm(b)$ and hence 
$a_\lm\in I$ for each $\lm\in\cF$. In particular, the projection $a_1$ is in $I$. Since 
$b$ is primitive, $a_1=\phi_b(a)b$ is a scalar multiple of $b$.  Hence, if $a_1 \neq 0$ 
then $b \in I$. This motivates the following construction.

\begin{definition}
Let $A$ be a primitive axial algebra.  We define the \emph{projection graph} $\Gamma$ 
to be the directed graph with vertex set $X$ and a directed edge from $a$ to $b$ if the 
projection $a_1$ of $a$ onto $b$ is non-zero. That is, if $\phi_b(a)\neq 0$.
\end{definition}

Given a directed graph $\Gamma$, the \emph{out set} $Out(\Gamma, Y)$ of a subset of 
vertices $Y$ is the set of all the vertices $v$ reachable from $Y$ by a directed path 
from $x \in Y$ to $v$.

The following lemma follows from the discussion above.

\begin{lemma}\label{orbitprojgraph}
Let $A$ be a primitive axial algebra and $\Gamma$ be its projection graph. If $Y$ is a 
set of axes contained in an ideal $I$ then $Out(\Gamma, Y)$ is also fully contained in 
$I$.
\end{lemma}

Recall that a directed graph $\Gamma$ is strongly connected if every vertex is reachable 
by a directed path from any other.  

\begin{corollary}\label{projstronglycon}
Let $A$ be a primitive axial algebra with a strongly connected projection graph. Then 
every proper ideal of $A$ is contained in the radical.
\end{corollary}

Recall from Corollary \ref{idealGinvariant} that every ideal is invariant under the 
Miyamoto group $G$. Hence, as a further improvement, we may quotient out by the action 
of $G$ to form the quotient graph $\bar \Gamma := \Gamma/G$.  It has as vertices orbits 
of axes with a directed edge from $a^G$ to $b^G$ if there exist axes $a' \in a^G$ and 
$b' \in b^G$ such that the projection $a'_1$ of $a'$ onto $b'$ is non-zero.  We call 
$\bar \Gamma$ the \emph{orbit projection graph}.

\begin{corollary}\label{orbprojstronglycon}
Let $A$ be a primitive axial algebra with a strongly connected orbit projection graph. 
Then every proper ideal of $A$ is contained in the radical.
\end{corollary}

We now consider the properties of the projection graph $\Gamma$ when there is a 
Frobenius form.

\begin{lemma}\label{frobprojlem}
Let $A$ be a primitive axial algebra that admits a Frobenius form.  Suppose that 
$(a,a) \neq 0 \neq (b,b)$ for $a,b \in X$.  The following are equivalent:
\begin{enumerate}
\item[$1.$] There is a directed edge $a \rightarrow b$ in $\Gamma$.
\item[$2.$] There is a directed edge $a \leftarrow b$ in $\Gamma$.
\item[$3.$] $(a,b) \neq 0$.
\end{enumerate}
\end{lemma}

\begin{proof}
By Lemma \ref{frobprim}, $(a,b) = \phi_a(b) (a,a)$, where the projection 
$b_1 = \phi_a(b) a$. Since the form is symmetric, the result follows.
\end{proof}

In light of the above result, when $A$ is a primitive axial algebra that admits a 
Frobenius form that is non-zero on the axes, we may consider $\Gamma$ to be an 
undirected graph.


\subsection{Uniqueness of the Frobenius form}

The same concept of the projection graph is useful when we want to establish uniqueness 
of the Frobenius form.

\begin{lemma}\label{uniquefrob}
A Frobenius form on a primitive axial algebra $A$ is uniquely determined by its values 
$(a,a)$ for $a \in X$.
\end{lemma}

\begin{proof}
If two Frobenius forms have the same values of $(a,a)$ for all $a\in X$ then their 
difference (which also associates with the algebra product) satisfies $(a,a)=0$ for all 
$a\in X$. Hence it suffices to show that the latter condition forces the form to be zero.

Clearly, $A$ is spanned by products of axes and so we just need to show that $(u,v)=0$ 
for all $u$ and $v$ that are products of axes. We use induction on the length of the 
products of axes for $v$.  If $v$ has length one, it is itself an 
axis in $X$. By Lemma \ref{frobprim}, $(u,v)=\phi_v(u)(v,v)=0$. Suppose now that $v$ has 
length at least two, which means that we may write $v=v_1v_2$, where $v_1$ and $v_2$ are 
shorter products. Then $(u,v)=(u,v_1v_2)=(uv_1,v_2)$. By induction, the latter value is 
zero. 
\end{proof}

In particular, for the form to be non-zero, at least one value $(a,a)$ must be non-zero. 
Clearly, we can scale the form so that $(a,a)$ takes any non-zero value we like, say 
$(a,a)=1$. By Lemma \ref{frobprim}, $\phi_a(b)(a,a)=(b,a)=(a,b)=\phi_b(a)(b,b)$. If 
$\phi_b(a)\neq 0$, we can deduce $(b,b)=\frac{\phi_a(b)}{\phi_b(a)}(a,a)$; that is, the 
value of $(b,b)$ can be determined from the value of $(a,a)$.

Recall that in the projection graph $\Gamma$ on $X$ we have a directed edge 
from $a$ to $b$ exactly when $\phi_b(a)\neq 0$. Hence the known values on a subset 
$Y\subset X$ allow us to deduce all values on the out set $Out(\Gamma,Y)$. In particular, 
we have the following.

\begin{proposition} \label{uniqueform}
If the projection graph $\Gamma$ of a primitive axial algebra $A$ is strongly 
connected then the Frobenius form on $A$, if it exists, is unique up to scaling.
\end{proposition}

The equation $(b,b)=\frac{\phi_a(b)}{\phi_b(a)}(a,a)$ means also that, for the 
Frobenius form to be a projection form (up to scaling), we must have 
$\frac{\phi_a(b)}{\phi_b(a)}=1$ for every edge of $\Gamma$; that is, 
$\phi_a(b)=\phi_b(a)$. Note that this condition may not be satisfied. For example, 
recent work Joshi on double axes in Matsuo algebras \cite{vijay} unearthed examples of 
axial algebras with fusion law $\cM(2\eta,\eta)$, where the unique Frobenius form is not 
a projection form.

Let us now discuss when the Frobenius form on $A$ is invariant under the Miyamoto group 
$G(X)$. Clearly, this requires that $(a^g,a^g)=(a,a)$ for all $a\in X$. It turns out this 
condition is also sufficient.

\begin{proposition}\label{invariantfrob}
The Frobenius form $(\cdot,\cdot)$ is invariant under the action of $G(X)$ if and only if 
$(a^g,a^g)=(a,a)$ for all $a\in X$ and $g\in G(X)$.
\end{proposition}

\begin{proof}
We have already mentioned that if the form is $G(X)$-invariant then $(a^g,a^g)=(a,a)$ for 
all $a\in X$ and $g\in G(X)$. Conversely, suppose that $(a^g,a^g)=(a,a)$ for all $a\in X$ 
and $g\in G(X)$. Fixing $g$, define a second form $(\cdot,\cdot)'$ by 
$(u,v)':=(u^g,v^g)$. It is straightforward to check that $(\cdot,\cdot)'$ is bilinear and, 
furthermore, Frobenius. Since $(a,a)'=(a^g,a^g)=(a,a)$ for each $a\in X$, we deduce from 
Proposition \ref{uniquefrob} that $(u,v)'=(u,v)$ for all $u,v\in A$. That is, 
$(u^g,v^g)=(u,v)$, proving that the form is $G(X)$-invariant.  
\end{proof}

Finally, if we are only interested in $G(X)$-invariant Frobenius forms then the uniqueness 
of such form can be checked via the orbit projection graph. 

\begin{proposition}
Let $A$ be a primitive axial algebra with a strongly connected orbit projection graph. 
Then a $G(X)$-invariant Frobenius form on $A$, if it exists, is unique up to scaling.
\end{proposition}


\subsection{An application: Norton-Sakuma algebras}

The $2$-generated primitive axial algebras of Monster type with a Frobenius form are 
well-known and have been completely classified.  There are nine such algebras, known as 
\emph{Norton-Sakuma} algebras \cite{hrs}. They all arise in the Griess-Norton algebra and 
their isomorphism type can be determined by the conjugacy class of $\tau_a\tau_b$, where 
$a$ and $b$ are two axes which generate the algebra. For this reason, they are usually 
labelled $1\A$, $2\A$, $2\B$, $3\A$, $3\C$, $4\A$, $4\B$, $5\A$ and $6\A$.  For a full 
description of these see, for example, \cite{ivanov}. We note that the known Frobenius 
forms on Norton-Sakuma algebras are inherited from the Griess-Norton algebra and, as such, 
they are positive definite and invariant under the respective Miyamoto groups.

\begin{proposition}
All the Norton-Sakuma algebras, except $2\B$, are simple and have unique Frobenius form (up to scaling).
\end{proposition}

\begin{proof}
It follows from the table on page 213 in \cite{ivanov} that $\phi_a(b)=0$ if and only if $a$ and 
$b$ generate a $2\B$ algebra; that is, $ab=0$. In particular, for the algebras $2\A$, 
$3\A$, $3\C$, $4\B$, $5\A$, and $6\A$ the projection graph is a complete (unoriented) 
graph. For the algebra $4\A$, the projection graph is the complete graph $K_4$ minus 
a matching; that is, a $4$-cycle. Hence for all these algebras the projection graph 
is connected (and hence strongly connected). 

On the one hand, Corollary \ref{projstronglycon} now implies that every proper ideal is 
contained in the radical, and the latter is trivial by Corollary \ref{Majalg}. Hence 
the algebra is simple.

On the other hand, Proposition \ref{uniqueform} tells us that the Frobenius form is unique 
up to scaling.
\end{proof}

Note that $2\B\cong\mathbb{R}\oplus\mathbb{R}$ and so it is not simple and, furthermore, 
the values $(a,a)$ and $(b,b)$ for the primitive axes in this algebra can be chosen 
arbitrarily. (And so the Frobenius form is definitely not unique up to scaling.)


\section{Sum decompositions}\label{sec:decomp}

If our definition of radical is good then we can expect that axial algebras with a 
trivial radical are semisimple, that is, direct sums of simple axial algebras. 
Hence it is natural to discuss here (direct) sum decompositions of axial algebras.

\subsection{Sums of algebras}

Suppose $A$ is a commutative non-associative algebra and $A_1,\ldots,A_n$ are 
subalgebras of $A$. 

\begin{definition}
An algebra $A$ is a \emph{sum} of subalgebras $\{ A_i : i \in I \}$, for some countable index set $I$, if $A_iA_j=0$ 
for all $i\neq j$ and $A=\la\la A_i : i \in I \ra\ra$.
\end{definition}

First of all, let us note the following.

\begin{lemma}
If $A$ is a sum of subalgebras $\{ A_i : i \in I \}$, then $A$ is the sum of the $A_i$ 
viewed as subspaces of $A$.
\end{lemma}

\begin{proof}
We denote by $\sum_{i \in I} A_i$, the vector space sum of the $A_i$.  We must show that it is the whole of $A$.
Taking two elements $u= \sum_{i \in I} u_i$ and $v=\sum_{i \in I} v_i$ of the subspace 
$\sum_{i \in I} A_i$, we see that $uv=(\sum_{i \in I} u_i)(\sum_{j \in I} v_j)= \sum_{i \in I} u_iv_i$, 
since all other pairwise products are zero. Hence $\sum_{i \in I} A_i$ is closed 
with respect to multiplication and so it is a subalgebra. Since it also contains all $A_i$, we 
conclude that $\sum_{i \in I} A_i$ coincides with $\la\la A_i : i \in I \ra\ra=A$.
\end{proof}

In light of the above result, from now on, we will write $A = \sum_{i \in I} A_i$ when $A$ is a sum of subalgebras $\{ A_i : i \in I \}$.  In particular, every element $u\in A$ can be written as $u=\sum_{i \in I} u_i$, where 
$u_i\in A_i$ for all $i$, and multiplication is given by 
$uv=(\sum_{i \in I} u_i)(\sum_{i \in I} v_i)=\sum_{i \in I} u_iv_i$. As usual, if the 
decomposition $u=\sum_{i \in I} u_i$ is unique for each $u\in A$, we call $A$ the 
\emph{direct sum}  of the subalgebras $A_i$ and write $A=\bigoplus_{i \in I} A_i$. In this 
case, $A$ is isomorphic to the \emph{external direct sum} defined as the Cartesian 
product $A_1\times\ldots\times A_n$ taken with the entry-wise operations.

Recall the following standard definition.

\begin{definition}
The \emph{annihilator} of a commutative algebra $A$ is
\[
\Ann(A):=\{u\in A : uA=0\}.
\] 
\end{definition}

Manifestly, $\Ann(A)$ is an ideal.  Returning to the axial algebra case, recall that the radical is the largest ideal $R(A, X)$ of $A$ not containing any axes $x \in X$.

\begin{lemma}
For an axial algebra $A$, $\Ann(A)\subseteq R(A,X)$.
\end{lemma}
\begin{proof}
Clearly, $\Ann(A)$ does not contain any axes in $X$ as $a \cdot a = a \neq 0$ for $a \in X$.
\end{proof}

Note that the annihilator does not necessarily equal the radical of an axial algebra.

\begin{example}
Recall the Matsuo algebra for the group $S_5$ from Example \ref{ex:Matsuo}.  If $\eta = -\frac{1}{3}$, then the radical is spanned by
\[
\sum_{a \in D} a
\]
which is easy to check is also in the annihilator.  So for $\eta = -\frac{1}{2}$, $R(A, X) = \Ann(A)$.

However, for $\eta = -2$, the radical is $4$-dimensional and is spanned by elements of the form
\[
(i,j) + (i,k) + (i,l) - (m, j) - (m, k) - (m, l)
\]
where $\{i,j,k,l,m\} = \{1,2,3,4,5\}$.  However, a simple calculation shows that such a vector is not in the annihilator.  Furthermore, these vectors span an irreducible submodule, so this implies the annihilator must be trivial.  Hence, for $\eta=-2$, $0 = \Ann(A) \subsetneqq R(A, X)$.
\end{example}

\begin{proposition}\label{annihilator}
If $A=\sum_{i \in I} A_i$, then 
\begin{enumerate}
\item[$1.$] $A_i\cap(\sum_{j\neq i} A_j)\subseteq\Ann(A_i)$
\item[$2.$] $\Ann(A)=\sum_{i \in I}\Ann(A_i)$
\end{enumerate}
\end{proposition}

\begin{proof}
Suppose that $u\in A_i\cap(\sum_{j\neq i} A_j)$ and let $a \in A$.  We may decompose $a = \sum_{j \in I} a_j$.  We have $ua = \sum_{j \in I} ua_j = ua_i + \sum_{j\neq i} ua_j $.  Since $u \in \sum_{j\neq i} A_j$, we see that $ua_i = 0$.  On the other hand, $u$ is in $A_i$ and hence $\sum_{j\neq i} ua_j = 0$ too.  Therefore $uA = 0$ and $u \in \Ann(A) \cap A_i \subseteq \Ann(A_i)$.
	
Since $A_iA_j = 0$ for all $j \neq i$, $\Ann(A_i)\subseteq\Ann(A)$.  Hence \allowbreak $\sum_{i \in I}\Ann(A_i)\subseteq\Ann(A)$. Conversely, if $u=\sum_{j \in I} u_j\in\Ann(A)$ 
then, taking $v\in A_i$, we get that $0=uv=(\sum_{j \in I} u_j)v=\sum_{j \in I} u_jv=u_iv$. 
So $u_iv=0$ for all $v\in A_i$; that is, $u_i\in\Ann(A_i)$. Therefore, 
$\Ann(A)\subseteq\sum_{i \in I}\Ann(A_i)$ and so we have equality.
\end{proof}

Suppose we take two different decompositions of an element $u$ and consider how these can differ.

\begin{lemma}\label{different}
Suppose $A=\sum_{i \in I} A_i$ and $u\in A$. For any two decompositions 
$u=\sum_{i \in I} u_i=\sum_{i \in I} u'_i$ of $u$, the difference $d_i=u_i-u'_i$ lies 
in $\Ann(A_i)$ for each $i$.
\end{lemma}
\begin{proof}
Clearly, $d_i\in A_i$. On the other hand, $d_i=\sum_{j\neq i}(u'_j-u_j)=-\sum_{j\neq i} d_j\in\sum_{j\neq i} A_j$. 
So $d_i\in A_i\cap\sum_{j\neq i} A_j$.  By Proposition \ref{annihilator}, $d_i \in \Ann(A_i)$.
\end{proof}

In particular, the following is true.

\begin{corollary}
If $A=\sum_{i \in I} A_i$ and $\Ann(A)=0$ then $A=\bigoplus_{i \in I} A_i$.
\end{corollary}

Recall that, for an axial algebra $A$, $\Ann(A)\subseteq R(A,X)$ and hence the assumption that the annihilator is trivial is satisfied when the radical $R(A,X)$ is trivial.

\subsection{Idempotents} \label{idempotents}

Axial algebras are generated by idempotents. So let us take a look at 
idempotents in sums of algebras.

\begin{lemma}
Suppose $A=\sum_{i \in I} A_i$ and $a\in A$ is an idempotent. Then 
$a$ admits a decomposition $a=\sum_{i \in I} a_i$, where every $a_i\in A_i$ is 
an idempotent.
\end{lemma}

\begin{proof}
Consider first an arbitrary decomposition $a=\sum_{i \in I} a'_i$ and set 
$a_i=(a'_i)^2$. Note that $a=a^2=(\sum_{i \in I} a'_i)^2=\sum_{i \in I} (a'_i)^2=
\sum_{i \in I} a_i$. So we have a decomposition. By Lemma \ref{different}, 
$d_i=a_i-a'_i\in\Ann(A_i)$. Therefore, $a_i^2=(a'_i+d_i)^2=(a'_i)^2+2a'_id_i+d_i^2=
(a'_i)^2=a_i$. Hence each $a_i$ is indeed an idempotent.
\end{proof}

Recall that we call an axis $a$ primitive when the $1$-eigenspace of $\ad_a$ 
coincides with $\la a\ra$. Similarly, we call a non-zero idempotent $a\in A$ 
\emph{primitive} when the $1$-eigenspace of $\ad_a$ is $1$-dimensional.

\begin{lemma} \label{only one}
Let $A=\sum_{i \in I} A_i$.  Than
\begin{enumerate}
\item[$1.$] Every idempotent of $A$ is contained in at most one $A_i$.
\item[$2.$] Every primitive idempotent is contained in exactly one $A_i$.
\end{enumerate}
\end{lemma}

\begin{proof}
First of all, note that a non-zero idempotent cannot lie in two summands. 
Indeed, if $a\in A_i$ and $a\in A_j$ with $i\neq j$ then 
$a\in A_i\cap A_j\subseteq\Ann(A_i)$. Hence $a=a^2=0$; a contradiction.

Write $a=\sum_{i \in I} a_i$, where every $a_i\in A_i$ is an idempotent. 
Note that $aa_i=(\sum_{j \in I} a_j)a_i=a_ia_i=a_i=1a_i$. Hence all $a_i$ are 
contained in the $1$-eigenspace of $\ad_a$. By primitivity, if two components, 
$a_i$ and $a_j$, are non-zero then $a_i=\lm a_j$ for some $\lm\in \mathbb{F}^\times$.  Then, $a_i \in A_i \cap A_j$ and so by the first part, $a_i = 0$, a contradiction.
\end{proof} 

\begin{theorem}\label{sumaxial}
Suppose that $A=\sum_{i \in I} A_i$ is a primitive axial algebra generated by a set of axes $X$. Let $X_i$ be the set of all those $a\in X$ that are contained 
in $A_i$ and let $B_i=\la\la X_i\ra\ra$. Then $A=\sum_{i \in I} B_i$. 
\end{theorem}

\begin{proof}
By Lemma \ref{only one}, every axis from $X$ lies in one and only one set $X_i$; that is, the 
sets $X_i$ form a partition of $X$.

Clearly, for $i\neq j$, we have $B_iB_j\subseteq A_iA_j=0$. So we just need to 
show that the subalgebras $B_i$ generate $A$. Since $X$ generates $A$, the algebra 
is spanned by all products of axes. Hence it suffices to show that each product 
is contained in some $B_i$. Clearly, if all axes involved in a product are from 
the same part $X_i$ then the product lies in $B_i$. Hence we just need to consider 
the case where the product $w$ involves axes from two different parts $X_i$ and $X_j$. 
In this case we will show that the product is zero by induction 
on the length of the product. Clearly the length of $w$ is at least two, 
and so we have $w=w_1w_2$, where $w_1$ and $w_2$ are shorter products. If, say, 
$w_1$ involves axes from two different parts then $w_1=0$ by induction and so $w=0$. 
Hence we can assume that $w_1$ only contains axes from one part, say $X_i$. Similarly, 
we can assume that $w_2$ only contains axes from $X_j$. However, this means that 
$w_1\in B_i$ and $w_2\in B_j$, and so $w=w_1w_2\in B_iB_j=0$. So indeed every product 
lies in some summand $B_i$ and so the subalgebras $B_i$ generate (in fact, span) $A$.
\end{proof}

This means that if an axial algebra decomposes as a sum, it also decomposes as a sum 
of smaller axial algebras. Furthermore, the summands come from partitions 
of the generating set $X$ satisfying $X_iX_j=0$ for all $i\neq j$.

\begin{theorem}
Suppose that $A$ is a primitive axial algebra such that $A = \la \la X \ra \ra = \la \la Y \ra \ra$ for two different generating sets of axes $X$ and $Y$ and $A$ has a decomposition $A=\sum_{i \in I} A_i$.  Let $X_i = A_i \cap X$ and $Y_i = A_i \cap Y$ and define $B_i = \la \la X_i \ra\ra$ and $C_i = \la\la Y_i \ra\ra$ as the axial algebras generated by the $X_i$ and $Y_i$ respectively.  Then, $B_i = C_i$ for all $i \in I$.
\end{theorem}
\begin{proof}
By Theorem \ref{sumaxial}, $A$ has a decomposition $A=\sum_{i \in I} B_i$ and another decomposition $A=\sum_{i \in I} A_i$.  By Lemma \ref{only one}, each axis $y \in Y$ is contained in a unique $A_i$ and a unique $B_i$.  However, since $A_i$ is a subalgebra, $B_i = \la\la X_i \ra\ra \leq A_i$ for all $i \in I$.  So, for each each $y \in Y$ there exists a unique $i \in I$ such that $y \in B_i \leq A_i$.  Hence, $C_i = \la\la Y_i \ra\ra \leq B_i$ and by symmetry the result follows.
\end{proof}

\begin{corollary}
The decomposition of a primitive axial algebra into a sum of axial subalgebras is stable under 
arbitrary change of axes.
\end{corollary}

\section{The non-annihilating graph $\Dl(X)$}\label{sec:NAgraph}

We can view the results above in a graph-theoretic way.

\begin{definition}
The \emph{non-annihilating graph} $\Dl(X)$ has vertex set $X$ and an edge $a \sim b$ between $a \neq b$ if $ab\neq 0$.
\end{definition}

Such a graph was introduced for axial algebras of Jordan type in \cite{HSS}.  In the case of Matsuo algebras, $\Delta(X)$ is also the non-commuting graph of the transpositions $X$.  For axial algebras of Monster type which admit a Frobenius form which is non-zero on each axis (this is all known examples), the non-annihilating graph is the same as the projection graph introduced in Section \ref{sec:radform}.

Suppose that $A = \sum_{i \in I} A_i$.  Then by Theorem \ref{sumaxial}, we may partition $X$ into a union of $X_i$ and $A = \sum_{i \in I} B_i$, where each $B_i$ is an axial algebra generated by $X_i$.  In particular, if $a\in X_i$ and $b\in X_j$, $i\neq j$, then $ab\in B_iB_j=0$. This means that each $X_i$ 
is a union of connected components of $\Dl(X)$.  It seems natural to ask: is it not true that the finest sum decomposition of $A$ arises when each $X_i$ is just a single 
connected component of $\Dl$?  For the Monster fusion law, we conjecture that this is indeed the case:

\begin{conjecture} \label{conj}
The finest sum decomposition of an axial algebra $A$ of Monster type arises when each $X_i$ is just a single connected component of $\Dl$.
\end{conjecture}

Equivalently, set $X_i$ to be the $i$th connected 
component of $\Dl(X)$. Then certainly $ab=0$ for $a\in X_i$, $b\in X_j$, whenever 
$i\neq j$. Define $A_i=\la\la X_i\ra\ra$. The above conjecture means that $A$ decomposes as a sum of the $A_i$.   The argument as in Theorem \ref{sumaxial} above 
shows that the $A_i$ generate $A$. What is missing is the claim that $A_iA_j=0$ for 
$i\neq j$. 

For axial algebras of Jordan type $\eta$ (those with fusion law $\cJ(\eta)$), the above conjecture holds and is Theorem A in \cite{HSS}.  We note that axial algebras of Jordan type are $1$-closed and their fusion law is Seress.

While we do not have any examples to the contrary, we cannot prove Conjecture \ref{conj} in full 
generality. We give a partial result, but before that we show that the groups behave well with respect to the finest sum decomposition.

Before we do so, we make an observation.  So far we have completely ignored the fusion law $\cF$ for $A$.  However, if $\Delta(X)$ does have more than one component, then in particular there exists two axes $a,b \in X$ such that $ab=0$.  So at the very least we must have that $0 \in \cF$.

\subsection{Miyamoto group}

Suppose that our axial algebra $A$ is $T$-graded, so that it has a Miyamoto group $G$. Recall 
that $\cF_t$ denotes the part of the grading partition corresponding to $t\in T$. For example, 
$\cF_{1_T}$ is the part corresponding to the identity element $1_T\in T$. We always have that 
$1\in\cF_{1_T}$.

\begin{lemma}\label{Tacomm}
Let $a,b\in X$ such that $ab=0$. Then $0\in\cF_{1_T}$ and $[T_a,T_b]=1$.
\end{lemma}

\begin{proof}
Suppose $0\in\cF_t$ for some $t\in T$. Then $0\ast 0\subseteq\cF_{t^2}$. On the other hand, Note $b\in A_0(a)$ and $b^2=b$. This means that $t^2=t$, and so $t=1_T$. We have shown that $0\in\cF_{1_T}$. 

Now, since $0\in\mathcal{F}_{1_T}$, $b\in A_0(a)$ is fixed by $\tau_a(\chi)$ for all $\chi\in T^*$. 
Therefore, $\tau_b(\chi')^{\tau_a(\chi)} = \tau_{b^{\tau_a(\chi)}}(\chi') = \tau_b(\chi')$ and 
hence $[T_a,T_b]=1$.
\end{proof}

\begin{theorem}
Let $A$ be a $T$-graded axial algebra and the components of $\Delta(X)$ be $X_i$ for $i\in I$.  
Then, $G(X)$ is a central product of its subgroups $G(X_i)$.
\end{theorem}

\begin{proof}
The Miyamoto group $G(X_i)$ is generated by all $T_a$ with $a\in X_i$.  By Lemma \ref{Tacomm}, 
$[T_a,T_b]=1$ for all $a\in X_i$, $b\in X_j$, $i\neq j$. Hence, every element of $G(X_i)$ 
commutes with every element of $G(X_j)$.  Since $G(X) =\la G(X_i): i\in I\ra$, it is 
a central product of the $G(X_i)$.
\end{proof}

So under the mild assumption that $0$ is in the trivially graded part, the finest sum decomposition of the non-annihilating graph induces a central product of the corresponding Miyamoto groups.

\subsection{Quasi-ideals}

Before we consider the algebra decomposition, we first introduce a new concept.

\begin{definition}
Suppose $A$ is an axial algebra generated by a set $X$ of axes. A \emph{quasi-ideal} 
in $A$ with respect to the generating set $X$ is a subspace $I\subseteq A$ such that 
$aI\subseteq I$ for all $a\in X$.
\end{definition}

Clearly, every ideal is a quasi-ideal. However, the converse is not true as even though $A$ is generated by a set $X$ of axes, it is non-associative.

The above definition of a quasi-ideal $I$ depends on a particular set of generating axes. Suppose the 
fusion law $\cF$ is $T$-graded. Since $I$ is invariant under each $\ad_a$, $a\in X$, 
Lemma \ref{invariant} implies that $I$ is invariant under the action of each $T_a$, 
and hence it is invariant under $G(X)$. Therefore, for every $b=a^g\in\bar X$, 
we have that $bI=a^gI^g\subseteq(aI)^g=I^g=I$ and so $I$ is also a quasi-ideal 
with respect to the closure $\bar X$ of $X$. We have the following.

\begin{proposition}\label{quasiwelldefined}
Let $I \subseteq A$ and $X$ and $Y$ be two sets of axes.
\begin{enumerate}
\item[$1.$] If $I$ is a quasi-ideal with respect to $X$, then it is invariant under the action of $G(X)$.
\item[$2.$] If $X \sim Y$, then $I$ is a quasi-ideal with respect to $X$ if and only if it is a quasi-ideal with respect to $Y$.  That is, being a quasi-ideal is stable.
\end{enumerate}
\end{proposition}

So the concept of quasi-ideals behaves well with respect to natural changes of generators.  We now introduce an important example of a quasi-ideal.

\begin{definition}
The \emph{spine} of an axial algebra $A$ is the quasi-ideal $Q(A, X)$ generated by all axes $X$.  If $Q(A, X) = A$, then we say that $A$ is \emph{slender}.
\end{definition}

It is clear that the spine contains the axes $X$ and is spanned by all products of the form $x_1(x_2( \dots(x_{k-1} x_k) \dots)$ where $x_i \in X$.  In particular, we have the following easy lemma.

\begin{lemma}
If $A$ is a $3$-closed axial algebra, then $A$ is slender.
\end{lemma}

\begin{proposition}
Let $A$ be an axial algebra and $X \sim Y$ be two equivalent sets of axes.  Then, $Q(A, X) = Q(A, Y)$.  That is, the spine of an axial algebra is stable.
\end{proposition}

\begin{proof}
It suffices to show equality of $Q(A,X)$ and $Q(A,\bar{X})$. Clearly, 
$Q(A,X)\subseteq Q(A,\bar{X})$. On the other hand, by Proposition \ref{quasiwelldefined}, 
$Q(A,X)$ is invariant under $G(X)$, which means that $\bar X\subset Q(A,X)$. Hence 
$Q(A,\bar X)\subseteq Q(A,X)$, and hence we have equality.
\end{proof}

\subsection{Algebras with Seress fusion laws}

As we noted before, $0 \in \cF$.  However, if $A$ were to have sum decompositions, then this imposes further constraints on $\cF$.  If an axis $a$ lies in the summand $A_i$ then every $A_j$, $j\neq i$, is contained in the $0$-eigenspace of $\ad_a$, since $aA_j=0$.  In particular, as $A_j$ is a subalgebra, $0\in 0\star 0$.  In order to show our partial result, we will, in fact, require a lot more than this.

\begin{definition}
The fusion law $\cF$ is \emph{Seress} if $0\in\cF$ and for any $\lm\in\cF$ we have 
$0\star \lm \subseteq\{\lm\}$.
\end{definition}

Note that for $1$, we already have that $1\star\lm\subseteq\{\lm\}$.  So, for Seress fusion laws, it follows that 
$1\star 0\subseteq\{1\}\cap\{0\}=\emptyset$. Also note that $0\star 0\subseteq\{0\}$ 
implies that $A_0(a)$ is a subalgebra for every axis $a$.

\begin{lemma}[Seress Lemma]\textup{\cite[Proposition 3.9]{hrs}}
If $\cF$ is Seress, then every axis $a$ associates with $A_1(a)+A_0(a)$. That is, for 
$x\in A$ and $y \in A_1(a)+A_0(a)$, we have that
\[
a(xy)=(ax)y.
\]
In other words, $\ad_a$ and $\ad_y$ commute.
\end{lemma}

\begin{proof}
Since the associativity identity is linear in $y$, we may consider $y \in A_1$ and $y \in A_0$ separately.  Associativity is also linear in $x$, so, since we may decompose $x$ with respect to $A = \bigoplus_{\lambda \in \mathcal{F}} A_\lambda(a)$, it suffices to check for $x \in A_\lambda$.  As $\mathcal{F}$ is Seress, $1 \star \lambda, 0 \star \lambda \subseteq \{ \lambda \}$ and so $xy \in A_\lambda$ for $y \in A_1$, or $y \in A_0$.  Hence,
\[
a(xy) = \lambda xy = (\lambda x)y = (ax)y. \qedhere
\]
\end{proof}

Suppose $A$ is generated by the set of axes $X=Y_1\cup Y_2$, where for all 
$a\in Y_1$ and $b\in Y_2$ we have $ab=0$. (We write $Y_1Y_2=0$.) Let $A_i=\la\la Y_i \ra\ra$.

\begin{theorem}\label{splitoff}
Let $A$ be an axial algebra with $X=Y_1\cup Y_2$ satisfying $Y_1Y_2=0$. If the fusion law is 
Seress, then $Q(A_1,X_1)$ annihilates $A_2$.
\end{theorem}

\begin{proof}
First of all, note that, for $x\in Y_1$, since $A_0(x)$ is a subalgebra and $Y_2\subseteq A_0(x)$, 
we have that $xA_2=0$. This means that $Y_1\subseteq U$, where $U:=\{u\in A_1 :  
uA_2=0\}=\Ann(A_2)\cap A_1$ is the annihilator of $A_2$ in $A_1$. Now, for $u\in U$ and $v\in A_2$, 
by Seress's Lemma, $(xu)v = x(uv) = x0 = 0$.  So, $xu\in U$. Since this is true for all $x\in Y_1$, 
$U$ is a quasi-ideal. Therefore, $U$ is a quasi-ideal containing $Y_1$, implying that 
$Q(A_1,Y_1)\subseteq U$.
\end{proof}

\begin{corollary} \label{induction}
Suppose that $A$ is an axial algebra such that $X=Y_1\cup Y_2$ is a disjoint union of axes and the 
fusion law is Seress.  If $A_1$ is slender, then $A = A_1 + A_2$.
\end{corollary}

We can now state our partial result for the conjecture about the non-annihilating graph.

\begin{theorem}\label{bodysum}
Let $A$ be an axial algebra with a Seress fusion law and let $X_i$ be the components of $\Delta(X)$ with $A_i = \la \la X_i \ra \ra$.  If all but possibly one $A_i$ are slender, then $A = \sum_{i \in I} A_i$.
\end{theorem}

\begin{proof}
This follows from Corollary \ref{induction} using induction on $|I|$.
\end{proof}

Recall that a $3$-closed axial algebra is slender, so the above theorem holds when all but at most one $A_i$ are $3$-closed.

As noted above, an axial algebra $A$ of Jordan type $\eta$ is $1$-closed and Seress.  So, by Corollary \ref{bodysum}, $A = \sum_{i \in I} A_i$, where $X_i$ are the connected components of $\Delta(A)$.  This is part (2) of Theorem A in \cite{HSS}.

The Ising fusion law $\cM(\alpha, \beta)$, of which the Monster fusion law is a special case, is also Seress.  Most of the examples we know for $\cM(\frac{1}{4}, \frac{1}{32})$ are $2$-closed, while a few are $1$- or $3$-closed \cite[Table 4]{algorithm}.  So we should expect the above decomposition theorem to apply to a wide class of examples.

However, there exist examples of axial algebras with fusion law $\cM(\frac{1}{4}, \frac{1}{32})$ 
that are not $3$-closed.  In \cite{algorithm}, we found an $18$-dimensional primitive axial algebra 
with Miyamoto group $S_3 \times S_3$, which is $4$-closed, but not $3$-closed.  In fact, in this 
example, $17=\dim(Q(A, X))<\dim(A)=18$.  So this algebra is not slender.

\end{document}